\documentclass[a4paper, 10pt]{article}

\usepackage{graphics,graphicx}
\usepackage{amssymb,amsmath}
\usepackage{enumerate}

\usepackage{color}

\newtheorem{theorem}{Theorem}

\newtheorem{claim}{Claim}[theorem]
\newtheorem{corollary}{Corollary}
\newtheorem{lemma}{Lemma}

 \def \sm {\setminus}
 \def \es {\emptyset}

\newenvironment{proof}[1][]%
{\noindent {\setcounter{equation}{0}\it Proof.
}{#1}{}}{\hfill$\Box$\vspace{2ex}}

\newenvironment{proof2}[1][]%
{\noindent {\setcounter{equation}{0}\it Proof.
}{#1}{}}{\hfill$\Diamond$\vspace{2ex}}

\begin{document}
\title{On graphs with no induced five-vertex path or  paraglider}

\author{
Shenwei Huang\thanks{College of Computer Science, Nankai University, Tianjin 300350, China.
This research is partially supported by the National Natural Science Foundation of China (11801284).}
\and
T. Karthick\thanks{Computer Science Unit, Indian Statistical
Institute, Chennai Centre, Chennai 600029, India. This research is partially supported by SERB-DST, Government of India under MATRICS scheme.}}
\date{\today}

\maketitle

\begin{abstract}

  Given two graphs $H_1$ and $H_2$, a graph is $(H_1,\,H_2)$-free if it contains no induced subgraph isomorphic to $H_1$ or $H_2$.
For a positive integer $t$, $P_t$ is the chordless path on $t$ vertices. A \emph{paraglider} is the graph  that consists of a chorless cycle $C_4$ plus a vertex adjacent to three vertices of the $C_4$.
   In this paper, we study the structure of ($P_5$,\,paraglider)-free graphs, and show that every such  graph $G$ satisfies $\chi(G)\le \lceil \frac{3}{2}\omega(G) \rceil$, where $\chi(G)$ and $\omega(G)$ are the chromatic number
and clique number of $G$, respectively.  Our bound is attained by the complement of the Clebsch graph on 16 vertices.
More strongly, we completely characterize all the ($P_5$,\,paraglider)-free graphs $G$ that satisfies $\chi(G)> \frac{3}{2}\omega(G)$.
We also construct an infinite family of ($P_5$,\,paraglider)-free graphs such that every graph $G$
in the family has $\chi(G)=\lceil \frac{3}{2}\omega(G) \rceil-1$. This shows that our upper bound is optimal
up to an additive constant and that there is no $(\frac{3}{2}-\epsilon)$-approximation algorithm to the chromatic number of ($P_5$,\,paraglider)-free graphs for any $\epsilon>0$.
\end{abstract}


\section{Introduction}
Graphs in this paper are simple and finite. Given a positive integer $\ell$, we denote the path on $\ell$ vertices by
$P_\ell$, and we denote the complete graph on $\ell$ vertices by $K_\ell$. For an integer $\ell\geq 3$, $C_\ell$ is the cycle on $\ell$ vertices.
A \emph{paraglider} is the graph   that consists of a $C_4$ plus a vertex
adjacent to three vertices of the $C_4$.    Given two graphs $G$ and $H$, we denote by $G\cup H$ the disjoint union of $G$ and $H$, and  by $G+H$ the join of $G$ and $H$.  The
union of $k$ copies of the same graph $G$ will be denoted by $kG$; for
example $2K_2$ denotes the graph that consists in two disjoint copies
of $K_2$. The complement of a graph $G$ is denoted by $\overline{G}$.  A \emph{hole} (\emph{antihole}) in a graph is an induced subgraph that is isomorphic to $C_\ell$ ($\overline{C_\ell}$) with $\ell \ge 4$, and $\ell$ is the
length of the hole (antihole).  A hole or an antihole is \emph{odd} if $\ell$ is odd.
Given a family of graphs ${\cal H}$, a graph $G$ is \emph{${\cal
F}$-free} if no induced subgraph of $G$ is isomorphic to a member of
${\cal F}$; when ${\cal F}$ has only one element $H$ we say that $G$
is $H$-free; when $\cal{F}$ has two elements $H_1$ and $H_2$, we simply write $G$ is ($H_1,H_2$)-free instead of $\{H_1,H_2\}$-free.

For any integer $k$, a \emph{$k$-coloring} of a graph $G$ is a mapping
$\psi:V(G)\rightarrow\{1,\ldots,k\}$ such that  $\psi(u)\neq \psi(v)$ whenever $u$ and $v$ are adjacent in $G$.  A graph is \emph{$k$-colorable}
if it admits a $k$-coloring.  The \emph{chromatic number} $\chi(G)$ of
a graph $G$ is the smallest integer $k$ such that $G$ is
$k$-colorable.
   A \emph{clique} in a graph
$G$ is a set of pairwise adjacent vertices, and the \emph{clique number} of $G$, denoted by $\omega(G)$, is the size of a maximum clique in $G$.
Obviously $\chi(H)\ge \omega(H)$ for every induced subgraph $H$ of $G$.
A graph $G$ is \emph{perfect} if every induced subgraph $H$ of $G$
satisfies $\chi(H) = \omega(H)$.  Chudnovsky et al. \cite{SPGT} showed that a graph is perfect   if and only if it does not contain an odd hole or an odd antihole as an induced subgraph, and is known as the Strong Perfect Graph Theorem (SPGT).
A class of graphs $\cal{G}$ is said   to be \emph{$\chi$-bounded} \cite{Gyarfas}
if there is a function $f$ (called a \emph{$\chi$-binding function}) such that  every  $G\in \cal{G}$
satisfies $\chi(G)\le f(\omega(G))$. For instance, the class of perfect graphs is $\chi$-bounded with identity function $f(x)=x$ as the $\chi$-binding function. In fact, several classes of graphs are known to be $\chi$-bounded; see \cite{KP, KZ, SS, survey}.

  Gy\'arf\'as \cite{Gyarfas} studied the $\chi$-boundedness for the class of $P_t$-free graphs, and showed that every $P_t$-free graph $G$ has $\chi(G) \leq  (t-1)^{\omega(G)-1}$. It is well known that for $t\leq 4$, $P_t$-free graphs are perfect.  The problem of
determining whether the class of $P_t$-free graphs ($t\geq 5$) admits a polynomial
$\chi$-binding function remains open,  and seems to be difficult even when $t=5$. Moreover, the existence of polynomial $\chi$-binding function for the class of $P_t$-free graphs ($t\geq 5$) would imply the Erd\"os-Hajnal conjecture for $P_t$-free graphs; see \cite{Chud-survey}. The best known $\chi$-binding
function $f$ for  the class of $P_5$-free graphs   satisfies $c(\omega^2/\log w)\le
f(\omega)\le 2^{\omega}$; see \cite{KPT-P5}.
Here we are interested in   $\chi$-binding functions for the class of  ($P_5$,\,$H$)-free graphs, for various graphs $H$. Recently, Brause et al. \cite{BRSV} showed that the class of ($2K_2$,\,$3K_1$)-free graphs does not admit a linear $\chi$-binding function. It follows that the class of ($P_5$,\,$H$)-free graphs, where $H$ is any $P_5$-free graph with independence number $\alpha(H) \geq 3$, does not  admit a linear $\chi$-binding function.  Thus it is interesting to the study  of $\chi$-boundedness for the class of ($P_5$,\,$H$)-free graphs where $\alpha(H) \leq 2$.
Choudum et al.\,\cite{CKS} showed that every ($P_5$,\,$C_4$)-free graph $G$ satisfies $\chi(G)\leq \lceil\frac{5\omega(G)}{4}\rceil$, and that every ($P_5$,\,$K_1+C_4$)-free graph $G$ satisfies $\chi(G)\leq 5\lceil\frac{5\omega(G)}{4}\rceil$.
It is shown in \cite{KM,survey} that every ($P_5$,\,diamond)-free graph $G$ satisfies $\chi(G)\leq \omega(G)+1$, and in \cite{BRSV} that every ($P_5$,\,paw)-free graph $G$ satisfies $\chi(G)\leq \omega(G)+1$.
Chudnovsky and Sivaram \cite{Chud-Siva} showed that every ($P_5$,\,$C_5$)-free graph $G$ satisfies $\chi(G) \leq 2^{\omega(G)-1}$. Fouquet
et al.~\cite{Fouquet} proved that there are infinitely many ($P_5,\,\overline{P_5}$)-free graphs $G$ with $\chi(G)\ge \omega(G)^{\mu}$, where $\mu= \log_25-1$, and that every ($P_5,\,\overline{P_5}$)-free graph $G$ satisfies $\chi(G)\le{{\omega(G)+1}\choose{2}}$.   Very recently, Chudnovsky et al. \cite{CKMM} showed that  every ($P_5$,\,$K_1+P_4$)-free graph $G$ satisfies $\chi(G)\leq \lceil\frac{5\omega(G)}{4}\rceil$. We refer to a recent comprehensive survey of Schiermeyer and Randerath \cite{survey} for more results.

In this paper, we study the structure of the class of  ($P_5$,\,paraglider)-free graphs, and show that every such graph $G$ satisfies $\chi(G)\leq \lceil\frac{3\omega(G)}{2}\rceil$.  Our bound is attained by the complement of the well-known $5$-regular \emph{Clebsch graph} on 16 vertices. More strongly, we completely characterize all the ($P_5$,\,paraglider)-free graphs $G$ that satisfies $\chi(G)> \frac{3}{2}\omega(G)$.
We also construct an infinite family of ($P_5$,\,paraglider)-free graphs such that every graph $G$
in the family has $\chi(G)=\lceil\frac{3}{2}\omega(G)\rceil-1$. This shows that our upper bound is optimal
up to an additive constant, and that there is no $(\frac{3}{2}-\epsilon)$-approximation algorithm  to the chromatic number of  ($P_5$,\,paraglider)-free graphs
for any $\epsilon>0$. Moreover, our results  generalizes the   results known on the existence of linear $\chi$-binding functions for ($P_5$,\,$C_4$)-free graphs,  ($P_5$,\,paw)-free graphs,  ($P_5$,\,diamond)-free graphs, and for ($3K_1$,\,paraglider)-free graphs \cite{CKS-2}.

\section{Notations and Preliminaries}

We use standard notation and terminology.
 In a graph $G$, the \emph{neighborhood} of a vertex $x$ is the set
$N_G(x)=\{y\in V(G)\setminus \{x\}\mid xy\in E(G)\}$; we drop the
subscript $G$ when there is no ambiguity.  The \emph{non-neighborhood} of a vertex $x$ is the set $V(G)\sm (N(x)\cup\{x\})$, and is denoted by $\overline{N(x)}$.   A vertex is
\emph{universal} if it is adjacent to all other vertices.  Two non-adjacent vertices $u$ and $v$ in a graph $G$ are \emph{comparable}
if $N(u) \subseteq N(v)$ or $N(v) \subseteq N(u)$.   For any $x\in V(G)$ and $A\subseteq
V(G)\setminus x$, we let $N_A(x) = N(x)\cap A$. Let $X$ be a subset of $V(G)$. We denote by $G[X]$
 the subgraph induced by $X$ in $G$. For simplicity, we write $G\sm X$ instead of $G[V(G)\sm X]$. Further if $X$ is singleton, say $\{v\}$, we write $G-v$ instead of $G\sm\{v\}$.   For any two subsets
$X$ and $Y$ of $V(G)$, we denote by $[X,Y]$, the set of edges that has
one end in $X$ and other end in $Y$.  We say that $X$ is
\emph{complete} to $Y$ or $[X,Y]$ is complete if every vertex in $X$
is adjacent to every vertex in $Y$; and $X$ is \emph{anticomplete} to
$Y$ if $[X,Y]=\emptyset$.  If $X$ is singleton, say $\{v\}$, we simply
write $v$ is complete (anticomplete) to $Y$ instead of writing $\{v\}$
is complete (anticomplete) to $Y$.     We
say that a subgraph $H$ of $G$ is \emph{dominating} if every vertex in
$V(G)\setminus V(H)$ is a adjacent to a vertex in $H$.  A
\emph{clique-cutset} of a graph $G$ is a clique $K$ in $G$ such that
$G \setminus K$ has more connected components than $G$. An \emph{atom} is a
connected graph without a clique-cutset.

 A
\emph{stable set} is a set of pairwise non-adjacent vertices.
  We say that two sets \emph{meet} if their
intersection is not empty.
In a graph $G$, we say that a stable set is \emph{good} if it meets
every clique of size $\omega(G)$.

An \emph{expansion} of a graph $H$ is any graph $G$ such that $V(G)$ can
be partitioned into $|V(H)|$ non-empty  sets
$Q_v$, $v\in V(H)$, such that $[Q_u,Q_v]$ is complete if $uv\in E(H)$,
and $[Q_u,Q_v]=\es$ if $uv\notin E(H)$. An expansion of a graph is a \emph{clique expansion} if each $Q_v$ is a clique,
 is a \emph{$\overline{P_3}$-free expansion} if each $Q_v$ induces a $\overline{P_3}$-free graph, and is a \emph{perfect expansion} if each $Q_v$ induces a perfect graph. By a classical result of Lov\'asz \cite{Lovasz}, any perfect expansion of a  perfect graph is perfect. In particular, any \emph{$\overline{P_3}$-free expansion} of a perfect graph is perfect.

\begin{figure}[h]
\centering
 \includegraphics[width=11cm]{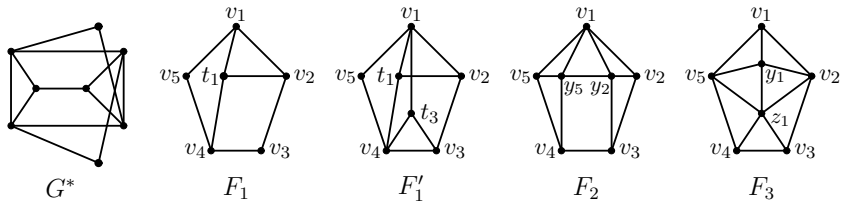}
\caption{Some special graphs}\label{fig:GF123}
\end{figure}

Let $G^*,F_1,F_1',F_2,F_3$ be five graphs as shown in
Figure~\ref{fig:GF123}.

\pagebreak

 Let $\cal{H}$ be the class of graphs $G$ such that $V(G)$ can be partitioned into five sets $Q_1, Q_2, R_1, R_2, S$ such that:
\begin{itemize}\itemsep=1pt
\item $Q_1=\{a_1,a_2\ldots, a_k\}$, $Q_2=\{b_1,b_2,\ldots,b_k\}$ (where $k\geq 2$),  $S$ are cliques, $[Q_1, Q_2]$ is a perfect matching, say
 $\{a_1b_1,a_2b_2,\ldots,a_kb_k\}$  and $|S|\leq k$.
\item $G[R_1]$ and $G[R_2]$ are perfect.
\item  $[Q_1, R_1]$, $[Q_2, R_2]$ are complete, $[Q_1\cup R_1, R_2] =\es$ and $[Q_2\cup R_2, R_1]=\es$.
\item  $[S, R_1\cup R_2]$ is complete.

\item There exists an injective function $f:S\rightarrow \{1,2,\ldots, k\}$ such that for each vertex  $x\in S$, $\{x\}$ is anti-complete to $\{a_{f(x)}, b_{f(x)}\}$, and is complete to $(Q_1\cup Q_2) \sm \{a_{f(x)}, b_{f(x)}\}$.
\item No other edges in $G$.
\end{itemize}
Clearly, the graphs $\overline{C_6}$ and $F_2$ belong to $\cal{H}$. See   Section~4 for more examples.

\medskip
We will use
the following theorem of Brandst\"adt and Ho\`ang \cite{BH}.
\begin{theorem}[\cite{BH}]\label{thm:BH}
Let $G$ be a ($P_5$,\,paraglider)-free atom that has no universal or pair of comparable vertices.  Then
either $G$ is $G^*$ or  every induced $C_5$ in $G$ is dominating. \hfill{$\Box$}
\end{theorem}

\section{Structure of ($P_5$,\,paraglider)-free graphs}
In this section, we prove the following structure theorem for the class of ($P_5$, paraglider)-free graphs.
\begin{theorem}\label{thm:P5P-free}
Let $G$ be a ($P_5$,\,paraglider)-free atom with no universal or pair of comparable vertices.
 Then one of the following hold:
\begin{itemize}\itemsep=0pt
\item $G$ is an induced subgraph of the complement of the Clebsch graph.
\item $G$ is a $\overline{P_3}$-free expansion of $C_5$.
\item $G$ has a stable set $S$ such that either $S$ is good or $G\sm S$ is perfect.
\item $G\in \cal{H}$.
\end{itemize}
\end{theorem}
\begin{proof} If $G$ is $G^*$, then $S:=\{v_7,v_8\}$ is a stable set such that $G\sm S \cong \overline{C_6}$ is perfect. If $G$ is perfect, then any color class in a $\chi(G)$-coloring of $G$ is a good stable set. So we may assume that $G$ is not $G^*$, and is not perfect. Now since
a $P_5$-free graph contains no hole of length at least $7$, and a paraglider-free graph
contains no antihole of length at least $7$, it follows by the Strong Perfect Graph Theorem \cite{SPGT} that $G$ contains a hole of length $5$. That is, $G$ contains a $C_5$ as an induced subgraph. Now the theorem follows from Theorem~\ref{thm:BH}, and from Theorems~\ref{thm:P5P-free-F1}, \ref{thm:P5PF1-free-F2}, \ref{thm:P5PF1F2-free-F3} and Theorem \ref{thm:P5PF123-free-C5} given below.
\end{proof}

In the next theorem, we make some general observations about the
situation when a ($P_5$,\,paraglider)-free graph contains a hole (which must
have length $5$).
\begin{theorem}\label{thm:C5}
Let $G$ be any ($P_5$,\,paraglider)-free graph that contains a $C_5$ with
vertex-set $C=\{v_1,\ldots,v_5\}$ and $\{v_iv_{i+1}\mid i\in
\{1,\ldots,5\}, i\bmod 5\}$. Suppose that $G$ is an atom and has no pair of comparable vertices. Let:
\begin{eqnarray*}
T_i &=& \{x\in V(G)\sm C\mid N_C(x)=\{v_{i}, v_{i+1}, v_{i+3}\}\}. \\
X_i &=& \{x\in V(G)\sm C\mid N_C(x)=\{v_{i-1}, v_{i+1}\}\}. \\
Y_i &=& \{x\in V(G)\sm C\mid N_C(x)=\{v_{i-1}, v_{i}, v_{i+1}\}\}.\\
Z_i &=& \{x\in V(G)\sm C\mid N_C(x)= C\setminus \{v_i\}\}.\\
A &=& \{x\in V(G)\sm C\mid N_C(x)=C\}.
\end{eqnarray*}
Moreover, let $T = T_1\cup\cdots\cup T_5$, $X =
X_1\cup\cdots\cup X_5$, $Y = Y_1\cup\cdots\cup Y_5$, and $Z = Z_1\cup\cdots\cup Z_5$. Then the following properties hold for all $i$, $i \bmod 5$:
\begin{enumerate}[(R1)]\itemsep=1pt
\item\label{V} $G[C]$ is a dominating induced subgraph of $G$ and $V(G) = C\cup A \cup T\cup X \cup Y \cup Z$.

\item\label{Ti}(a) $|T_i|\le 1$. If $|T_i|= 1$, then we denote $T_i$ by $\{t_i\}$.\\ (b) $[T_i, T_j]=\emptyset$, for every $j$; so $T$ is an independent set.\\    (c) $[T_i, X_{i+3}]$ is complete, and $[T_i, X_j]=\es$, for every $j\neq i+3$.

\item\label{Xi}  (a) $X_i$ is an independent set.\\ (b) $[X_i, X_{i+1}]$ is complete.\\  (c) $|[X_{i}, X_{i+2}]|\le 1$. \\(d) If $X_{i+1}\neq \es$, then $[X_{i}, X_{i+2}]=\es$.

\item\label{Yi} (a) $G[Y_i]$ is $(K_2\cup K_1)$-free. Hence $G[Y_i]$ is a complete multi-partite graph.\\ (b) $[Y_i, Y_{i+1}]$ is complete.\\ (c) If $[Y_i,Y_{i+2}]\neq \es$, then $[Y_i, Y_{i+2}]$ is a matching.\\ (d) If $Y_i\neq \es$, then $Y_{i-1}$ and $Y_{i+1}$ are cliques.\\ (e) If $y\in Y_{i}$ and if $[\{y\}, Y_{i+2}]$ is complete, then $|Y_{i+2}|\le 1$. More generally, if $[Y_i, Y_{i+2}]$ is complete, then $|Y_i|\le 1$ and $|Y_{i+2}|\le 1$.

\item\label{Zi} (a) $|Z_i|\le 1$. If $|Z_i|= 1$, then we denote $Z_i$ by $\{z_i\}$.\\ (b) $[Z_i, Z_{i+1}]=\es$.\\ (c) $[Z_i, Z_{i+2}]$ is complete.

 \item\label{YZX} (a) $[X_i, Y_i\cup Y_{i+1}\cup Y_{i-1}\cup (Z\sm Z_i)]$ is complete.\\ (b)  $[X_i, Y_{i-2}\cup Y_{i+2}\cup Z_i]=\es$.

\item\label{ZiYj} (a) $[Z_i, Y_i \cup Y_{i-2}\cup Y_{i+2}]$ is complete.\\ (b)  $[Z_i, Y_{i-1}\cup Y_{i+1}]=\es$.

\item\label{AneiA}  (a)  $A$ is a clique. \\ (b) $[A, V(G)\sm (A\cup Y)]$ is complete.

\item\label{Tneqemp} Suppose that $T_i\neq \es$. Then:\\ 
 (a) The sets $T_{i-1}\cup T_{i+1}$, $Y$, $Z\sm Z_{i+3}$ are empty.\\ (b) $[X_i, X_{i+2}]$, $[X_{i+1}, X_{i-1}]$ and $[T_i, Z_{i+3}]$ are empty.

\item\label{AYi} Let $x, y\in V(G)$ and $q\in A$. Then the following hold:\\ (a) If $x\in Y_i$ and $y\in Y_{i+1}\cup Y_{i+2}$ are adjacent, then $q$ is either complete or anti-complete to $\{x, y\}$. \\
    (b) If $x\in Y_i$ and $y\in Y_{i+2}$ are not adjacent, then $q$ is adjacent one of $x$, $y$.\\
(c) If $x\in Y_i$ and $y\in Y_{i}$ are  adjacent, then $q$ is adjacent to one of $x$,  $y$.\\
(d) $\overline{N(q)} \cap Y_i$ is a stable set.

\item\label{F1freeXemp} Suppose that $G$ is $F_1$-free. Then $[X_i, X_{i+2}\cup X_{i-2}\cup Y_{i+2}\cup Y_{i-2}] =\es$, and hence
 $X=\es$.

\end{enumerate}
\end{theorem}

\begin{proof} Let $G$ be the given graph with vertex-set $V$ and edge-set $E$.

 \medskip
\noindent{Proof of (R$\ref{V}$).} Since $G$ has no clique cut-set, by Theorem~\ref{thm:BH}, $G[C]$ is dominating, and so every vertex in $V\sm C$ has a neighbor in $C$. Now (R\ref{V}) follows since $G$ is $P_5$-free. Indeed if a vertex $x\in V\sm C$ has exactly one neighbor (say, $v_i$) or has exactly two neighbors that are consecutive (say, $v_i$ and $v_{i+1}$) in $C$, then $x$-$v_{i}$-$v_{i-1}$-$v_{i-2}$-$v_{i-3}$ is a $P_5$. \hfill{$\diamond$}

\medskip
\noindent{Proof of (R$\ref{Ti}$).}~$(a)$:~Otherwise, for any two vertices $x$ and $y$ in $T_i$, either $\{x,v_{i+1},$ $v_{i+2},v_{i+3},y\}$
or $\{x,v_i,y,v_{i+3},v_{i+1}\}$ induces a paraglider. So (a) holds.  \\
$(b)$:~Suppose to the contrary that there are adjacent vertices $x\in T_i$ and $y\in T_j$.  Now if $j\neq i-1$, then $\{x,v_{i+1},v_{i+2},v_{i+3},y\}$ induces a paraglider, and if  $j= i-1$, then $\{y,v_{i},v_{i+1},v_{i+2},x\}$ induces a paraglider, a contradiction. This proves item~$(b)$.\\
$(c)$:~Pick a vertex $x\in T_i$ and a vertex $y\in X_j$. Up to symmetry, we may assume that $j\in \{i,i+2,i+3\}$.
  If $j=i$, then $xy\notin E$, for otherwise $\{y,v_{i+1},v_i,v_{i-1},x\}$ induces a paraglider.
If $j=i+2$,  then $xy\notin E$, for otherwise $\{x,v_{i+1},v_{i+2},v_{i+3},y\}$ induces a paraglider.
If $j=i+3$, then $xy\in E$, for otherwise $v_i$-$x$-$v_{i+3}$-$v_{i+2}$-$y$ is a $P_5$.   Since this
holds for any $x$ and $y$, it proves item~(c). \hfill{$\diamond$}

\medskip
\noindent{Proof of (R$\ref{Xi}$).}~$(a)$:~Otherwise, for any two adjacent vertices $x$ and $y$ in $X_i$, $\{x,v_{i+1},v_{i},v_{i-1},y\}$
induces a paraglider. So (a) holds. \\
$(b)$:~Suppose not, and let $x\in X_i$ and $y\in X_{i+1}$ be not adjacent. Then $x$-$v_{i-1}$-$v_{i-2}$-$v_{i-3}$-$y$ is a $P_5$. So (b) holds. \\
$(c)$:~We may assume that $[X_i,X_{i+2}]\neq \es$. We first claim that $[X_i,X_{i+2}]$ is a matching. Suppose not. Then, up to symmetry, we may assume that there exist vertices $x\in X_i$ and $y,z\in X_{i+2}$ such that $xy,xz\in E$. By (a), $yz\notin E$. But then $\{x,z,v_{i+3},y,v_{i+1}\}$ induces a paraglider. So $[X_i,X_{i+2}]$ is a matching. Now, if $|[X_i,X_{i+2}]|\ge 2$, then there exist matching edges $e,f\in [X_i,X_{i+2}]$, say $e:=xy$ and
$f:=x'y'$ with $x,x'\in X_i$ and $y,y'\in X_{i+2}$. By (a), we have $xx'\notin E$ and $yy'\notin E$.  But then $y$-$x$-$v_{i-1}$-$x'$-$y'$ is a $P_5$, a contradiction. This proves item~(c). \hfill{$\diamond$}

\medskip
\noindent{Proof of (R$\ref{Yi}$).}~$(a)$:~ Suppose to the contrary that $G[Y_i]$ contains an induced $K_2\cup K_1$ with vertex-set $\{x,y,z\}$ and edge-set $\{xy\}$. Then $\{x, v_{i-1},z,v_{i+1},y\}$ induces a paraglider, which is a contradiction. So (a) holds.  \\
$(b)$:~Suppose not, and let $x\in Y_i$ and $y\in Y_{i+1}$ be not adjacent. Then $x$-$v_{i-1}$-$v_{i-2}$-$v_{i-3}$-$y$ is a $P_5$. So (b) holds. \\
$(c)$:~Suppose not. We may assume, up to symmetry, that $x\in Y_i$ and $y,z\in Y_{i+2}$ such that $xy,xz\in E$. Then $\{x,y,v_{i-2},v_{i-1},z\}$ or
$\{x,y,v_{i-2},z,v_{i+1}\}$ induces a paraglider, a contradiction. This proves item~(c).\\
$(d)$:~Let $x\in Y_i$. Suppose to the contrary  that there are non-adjacent vertices $y$ and $z$ in $Y_{i+1}$.  By (b), $xy,yz\in E$. But then $\{x,y,v_{i+2},z,v_i\}$ induces a paraglider which is a contradiction. So $Y_{i+1}$ is a clique. Likewise, $Y_{i-1}$ is a clique. This proves item~(d). \\
$(e)$:~ This follows by item~(c). \hfill{$\diamond$}

\medskip
\noindent{Proof of (R$\ref{Zi}$).}~$(a)$:~Otherwise, for any two vertices $x$ and $y$ in $Z_i$, either $\{v_i,v_{i+1},$ $x,v_{i-1},y\}$
or $\{x,v_{i+1},y,v_{i-1},v_{i+2}\}$ induces a paraglider.   \\
$(b)$:~Suppose not, and let $x\in Z_i$ and $y\in Z_{i+1}$ be   adjacent. Then $\{v_i,v_{i+1},$ $v_{i+2},y,x\}$ induces a paraglider.  \\
$(c)$:~Suppose not, and let $x\in Z_i$ and $y\in Z_{i+2}$ be not adjacent. Then $\{v_i,v_{i+1},$ $x,v_{i-1},y\}$ induces a paraglider.
 \hfill{$\diamond$}

\medskip
\noindent{Proof of (R$\ref{YZX}$).}~$(a)$:~Suppose not, and  let $x\in X_i$ and $y\in Y_i\cup Y_{i+1}\cup Y_{i-1}\cup (Z\sm Z_i)$ be non-adjacent. By symmetry, we may assume that $y\in Y_i\cup Y_{i+1}\cup Z_{i-1}\cup Z_{i+2}$. Now if $y\in Y_i\cup Z_{i+2}$, then $\{x,v_{i+1},y,v_{i-1},v_i\}$ induces a paraglider, and if $y\in Y_{i+1}\cup Z_{i-1}$, then $x$-$v_{i-1}$-$v_i$-$y$-$v_{i+2}$ is a $P_5$, a contradiction. This proves item~(a).\\
$(b)$:~Suppose not, and let $x\in X_i$ and $y\in Y_{i-2}\cup Y_{i+2}\cup Z_i$ be adjacent. By symmetry, we may assume that $y\in Y_{i+2}\cup Z_i$.
Now if $y\in Y_{i+2}$, then $v_i$-$v_{i-1}$-$x$-$y$-$v_{i+2}$ is a $P_5$, and if $y\in Z_i$, then $\{x,v_{i+1},v_i,v_{i-1},y\}$ induces a paraglider, a contradiction. This proves item~(b). \hfill{$\diamond$}

\medskip
\noindent{Proof of (R$\ref{ZiYj}$).}~$(a)$:~Suppose not. Up to symmetry, we may assume that there are non-adjacent vertices $x\in Z_i$ and $y\in Y_i\cup Y_{i+2}$. Now if $y\in Y_i$, then $\{v_i,v_{i+1},x,v_{i-1},y\}$ induces a paraglider, and if $y\in Y_{i+2}$, then $v_i$-$v_{i-1}$-$x$-$v_{i+2}$-$y$ is a $P_5$, a contradiction. This proves item~(a).\\
$(b)$:~Suppose not, and let $x\in Z_i$ and $y\in Y_{i+1}\cup Y_{i-1}$ be adjacent. Then $\{v_i,v_{i+1},x,v_{i-1},y\}$ induces a paraglider. \hfill{$\diamond$}

\medskip
\noindent{Proof of (R$\ref{AneiA}$).} Suppose not, and let $x\in A$ and $y\in A \cup (V(G)\sm Y)$ be non-adjacent.
 If $y\in A$, then $\{x,v_i,y,v_{i+3},v_{i+1}\}$ induces a paraglider. So let us assume that  $y\in V(G)\sm Y$.
 Then there exist $j\in \{1,\ldots,5\}$, $j$ modulo $5$~such that $yv_j,yv_{j+2}\in E$ and $yv_{j+1}\notin E$. But then
 $\{y,v_j,v_{j+1},v_{j+2},x\}$ induces a paraglider, a contradiction. \hfill{$\diamond$}

\medskip
\noindent{Proof of (R$\ref{Tneqemp}$).} Let $x\in T_i$.\\
$(a)$:~Suppose to the contrary that there exists a vertex $y\in T_{i+1}\cup T_{i-1}\cup Y\cup (Z\sm Z_{i+3})$.
First suppose that $y\in T_{i+1}\cup Y_i \cup Z_i$.  Then since $y$-$v_{i}$-$x$-$v_{i+3}$-$v_{i+2}$ or
$x$-$v_{i}$-$v_{i-1}$-$y$-$v_{i+2}$ is not a $P_5$, we have $xy\in E$. But then $\{v_i,v_{i+1},y,v_{i-1},x\}$ or $\{x,v_i,v_{i-2},v_{i-1},y\}$ induces a a paraglider.  So $y\notin T_{i+1}\cup Y_i \cup Z_i$. Likewise, $y\notin T_{i-1}\cup Y_{i+1}\cup Z_{i+1}$.
Next suppose that $y\in Y_{i+2}$. Then since $\{x,v_{i+1},v_{i+2},v_{i+3},y\}$ does not induce a paraglider, we have $xy\in E$. But then $v_{i-1}$-$v_i$-$x$-$y$-$v_{i+2}$ is a $P_5$. So  $y\notin Y_{i+2}$. Likewise, $y\notin Y_{i-1}$.
Next suppose that $y\in Y_{i-2}$. Then since $\{x,v_i,v_{i-1},v_{i-2},y\}$ does not induce a paraglider, $xy\notin E$. But then $x$-$v_{i}$-$v_{i-1}$-$y$-$v_{i+2}$ is a $P_5$. So $y\notin Y_{i-2}$. Finally, suppose that $y\in Z_{i+2}\cup Z_{i-1}$. Up to symmetry, we may assume that $y\in Z_{i+2}$. Then since $\{x,v_i,v_{i-1},v_{i-2},y\}$ does not induce a paraglider, we have $xy\in E$. But then $\{x,v_{i+1},v_{i+2},v_{i+3},y\}$ induces a paraglider, a contradiction. This proves item~(a).\\
$(b)$:~Suppose that there is an edge $yz$ in one of the listed sets.  If $y\in X_i$ and $z\in X_{i+2}$, then by (R\ref{Ti}:c), we have $xy,xz\notin E$; and then $z$-$y$-$v_{i-1}$-$v_i$-$x$ is a $P_5$.  If $y\in X_{i+1}$ and $z\in X_{i-1}$, then by (R\ref{Ti}:c), we have $xy,xz\notin E$; and then $x$-$v_{i+1}$-$v_{i+2}$-$y$-$z$ is a $P_5$. If $y\in T_i$ and $z\in Z_{i+3}$, then $\{y,v_i,v_{i-1},v_{i-2},z\}$ induces a paraglider. These contradictions show that (b) holds.   \hfill{$\diamond$}

\medskip
\noindent{Proof of (R$\ref{AYi}$).}~$(a)$:~Suppose not. Up to symmetry, we may assume that $qx\in E$ and $qy\notin E$. Then either $\{q,x,y,v_{i+2},v_i\}$ or $\{v_{i-1},x,y,v_{i-2},q\}$ induces a paraglider, a contradiction.  So (a) holds. \\
$(b)$:~Otherwise, $x$-$v_{i}$-$q$-$v_{i+2}$-$y$ is a $P_5$.\\
$(c)$:~Otherwise, $\{x,v_{i+1},q,v_{i-1},y\}$ induces a paraglider.\\
$(d)$:~This follows by item~(c).  \hfill{$\diamond$}

\medskip
\noindent{Proof of (R$\ref{F1freeXemp}$).}~Suppose to contrary that there are adjacent vertices $x\in X_i$ and $y\in X_{i+2}\cup X_{i-2}\cup Y_{i+2}\cup Y_{i-2}$. We may assume, up to symmetry, that $y\in X_{i+2}\cup Y_{i+2}$. Now $\{v_i,v_{i+1},y,v_{i-2},v_{i-1},x\}$ induces an $F_1$, a contradiction. So, $[X_i, X_{i+2}\cup X_{i-2}\cup Y_{i+2}\cup Y_{i-2}]=\es$.

Now we show that $X=\es$. Suppose to the contrary that $X\neq \es$ and let $x\in X$, say $x\in X_i$ for some $i$. We claim that $x$ and $v_i$ are comparable.  Since $G$ is $F_1$-free, $T=\es$. Now by the preceding point, by the definition of $X_i$, and by (R\ref{Xi}),(R\ref{YZX}), (R\ref{AneiA}:b), and since $G[C]$ is dominating, we see that $N_G(x)=\{v_{i+1},v_{i-1}\}\cup X_{i+1}\cup X_{i-1}\cup Y_i \cup Y_{i+1}\cup Y_{i-1}\cup (Z\sm Z_i)\cup A$, and $\overline{N_G(x)}=\{v_{i+2},v_{i-2}\}\cup (X_{i}\sm \{x\})\cup X_{i-2}\cup X_{i+2}\cup Y_{i+2}\cup Y_{i-2}\cup Z_i$.
So, $N_G(x)= N_G(v_i)$ and $\overline{N_G(x)}=\overline{N_G(v_i)}$, and hence we  conclude that $x$ and $v_i$ are comparable, a contradiction. So (R$\ref{F1freeXemp}$) holds. \hfill{$\diamond$}

This completes the proof of Theorem~\ref{thm:C5}.
\end{proof}

\begin{theorem}\label{thm:P5P-free-F1}
Let $G$ be a ($P_5$,\,paraglider)-free atom with no universal vertex.
Suppose that $G$ contains   $F_1$. Then $G$ has a stable set $S$ such that $G\sm S$ is a bipartite graph or a bull. In particular, $G\sm S$ is perfect.
\end{theorem}
\begin{proof} Let $G$ be the given graph with vertex-set $V$ and edge-set $E$.
First suppose that $G$ contains an $F_1'$.  Consider the graph $F_1'$ as shown in Figure~\ref{fig:GF123} and let
$C=\{v_1,\ldots, v_5\}$.  We use the same notation as in
Theorem~\ref{thm:C5} and use the properties in Theorem~\ref{thm:C5}.  Then by (R\ref{Ti}:a), $T_1 = \{t_1\}$ and $T_3=\{t_3\} $.
Moreover, by (R\ref{Tneqemp}:a), $T=\{t_1, t_3\}$, $Y=\es$ and $Z=\es$.
Then since $Y=\es$,   any vertex in $A$ is a universal vertex of $G$ (by (R\ref{AneiA})), and hence $A=\es$.
Also, by (R\ref{Tneqemp}:b), $[X_1,X_3]=\es$ and $[X_2,X_4]=\es$.
Now, let us define  $S:= \{t_1, t_3, v_5\} \cup X_5$, $S_1:=\{v_1,v_3\}\cup X_1\cup X_3$, and $S_2:=\{v_2,v_4\}\cup X_2\cup X_4$.
Then by (R\ref{Xi}:a)  and (R\ref{Ti}:c), the set $S:= \{t_1, t_3, v_5\} \cup X_5$ is a stable set. Also, by the preceding points and (R\ref{Xi}:a), we see that $V(G)\sm S = S_1\cup S_2$, and $S_1$ and $S_2$ are stable sets. Hence $G\sm S$ is bipartite.

Suppose that $G$ contains no $F_1'$. Consider the graph $F_1$ as shown in Figure~\ref{fig:GF123} and let
$C=\{v_1,\ldots, v_5\}$.  We use the same notation as in
Theorem~\ref{thm:C5} and use the properties in Theorem~\ref{thm:C5}.  Since $G$ has no $F_1'$, by (R\ref{Ti}:a), $T = \{t_1\}$.  Then by (R\ref{Tneqemp}), the sets $Y$, $Z\sm Z_4$, $[X_2,X_5]$ and $[\{t_1\}, Z_4]$ are empty.  Then since $Y=\es$,   any vertex in $A$ is a universal vertex of $G$ (by (R\ref{AneiA})), and hence $A=\es$. Also, if there are adjacent vertices $x_2\in X_2$ and $x_4\in X_4$, then $\{v_1,v_2,v_3,x_4,v_5,t_1,x_2\}$ induces an $F_2$.
So $[X_2,X_4]=\es$. Likewise, $[X_3,X_5]=\es$.

Suppose that $Z_4=\es$. Then let us define  $S:= \{t_1, v_3, v_5\}\cup X_3\cup X_5$, $S_1:=\{v_1\}\cup X_1$, and $S_2:=\{v_2,v_4\}\cup X_2\cup X_4$.
Then by (R\ref{Xi}:a)  and (R\ref{Ti}:c), the set $S:= \{t_1, t_3, v_5\} \cup X_5$ is a stable set. Also, by the preceding points and (R\ref{Xi}:a), we see that $V(G)\sm S = S_1\cup S_2$, and $S_1$ and $S_2$ are stable sets. Hence $G\sm S$ is bipartite.

 So let us assume that $Z_4\neq \es$, and by (R\ref{Zi}:a), $Z_4=\{z_4\}$. Then by (R\ref{Tneqemp}:b), $t_1z_4\notin E(G)$. Now we claim that $X_j=\es$, for $j\neq 4$. Suppose not. Up to symmetry, we may assume that there exists a vertex $x\in X_1\cup X_3$. If $x\in X_1$, then by (R\ref{Ti}:c) and (R\ref{YZX}:a), we have $t_1x\notin E$ and $xz_4\in E$. But then $v_4$-$t_1$-$v_1$-$z_4$-$x$ is a $P_5$.   If $x\in X_3$, then since $\{v_2,v_3,v_4,x,z_4\}$ does not induce a paraglider, $xz_4\notin E$. But then $x$-$v_4$-$v_3$-$z_4$-$v_1$ is a $P_5$. So, we conclude that $X_j=\es$, for $j\neq 4$.  Now, by (R\ref{Xi}:a) and (R\ref{YZX}:b), the set $S:=\{z_4,v_4\}\cup X_4$ is a stable set such that $G\sm S:=G[\{t_1,v_5,v_1,v_2,v_3\}]$ is a bull.  This completes the proof of the theorem.
\end{proof}

\begin{theorem}\label{thm:P5PF1-free-F2}
Let $G$ be a ($P_5$,\,$F_1$,\,paraglider)-free atom with no  universal or pair of comparable vertices. Suppose that $G$ contains $F_2$.  Then one of the following hold:
\begin{itemize}\itemsep=0pt
\item $G$ is an induced subgraph of the complement of the Clebsch graph.
\item   $G$ has a good stable set.
\item  $G \in \cal{H}$.
\end{itemize}
\end{theorem}
\begin{proof} Let $G$ be the given graph with vertex-set $V$ and edge-set $E$.
Consider the graph $F_2$ as shown in Figure~\ref{fig:GF123} and let
$C=\{v_1,\ldots, v_5\}$.  We use the same notation as in
Theorem~\ref{thm:C5} and use the properties in Theorem~\ref{thm:C5}.
 So $y_2\in Y_2$ and $y_5\in Y_5$, and $y_2y_5\in [Y_2, Y_5]$. Let $\{a_1b_1,a_2b_2,\ldots,a_kb_k\}$ denote the edges of $[Y_2,Y_5]$, and let $a_1b_1:= y_2y_5$. Moreover, let $Y_2^*:= Y_2\sm \{a_1\}, Y_2':= Y_2\sm \{a_1,a_2,\ldots,a_k\}, Y_5^*:= Y_5\sm \{b_1\}$ and $Y_5'=Y_5\sm \{b_1,b_2,\ldots,b_k\}$.

\smallskip
Since $G$ is $F_1$-free, $T=\es$, and by (R\ref{F1freeXemp}), $X=\es$. Then we have the following:

\begin{claim}\label{Y2Y5}For each $i\in \{1,2,\ldots,k\}$, the following hold:
\begin{enumerate}[(i)]
 \item Any vertex in $Y_2\sm \{a_i\}$ is adjacent to $a_i$ and non-adjacent to $b_i$ (and similarly, any vertex in $Y_5\sm \{b_i\}$ is adjacent to $b_i$ and non-adjacent to $a_i$). In particular, $\{a_1,\ldots,a_k\}$ and $\{b_1,\ldots,b_k\}$ are cliques.
      \item $[Y_3, \{b_i\}]$ is complete and $[Y_4, \{a_i\}]$ is complete.
\item $[Y_3, Y_5\sm \{b_i\}]=\es$ and $[Y_4, Y_2\sm \{a_i\}]=\es$.
 \item Any vertex $a\in A$ is either complete to $\{a_i, b_i\}$ or anti-complete to $\{a_i, b_i\}$. Moreover, if $a\in A$ and if there exists an index $i$ such that $a$ is anti-complete to $\{a_i, b_i\}$, then $a$ is complete to $(Y_2\sm \{a_i\}) \cup (Y_5\sm \{b_i\})$.
 \end{enumerate}
 \end{claim}
 \begin{proof2} $(i)$: Let $x\in Y_2\sm \{a_i\}$ be arbitrary. If $xa_i\notin E$, then since $\{x,v_1,a_i,v_3,b_i\}$ does not induce a paraglider, we have $xb_i\notin E$, and then $x$-$v_2$-$a_i$-$b_i$-$v_5$ is a $P_5$. So, $xa_i\in E$. Moreover, since $\{x,v_3,v_4,b_i,a_i\}$ does not induces a paraglider, we have $xb_i\notin E$. Thus any vertex in $Y_2\sm \{a_i\}$ is adjacent to $a_i$ and non-adjacent to $b_i$. Likewise, any vertex in $Y_5\sm \{b_i\}$ is adjacent to $b_i$ and non-adjacent to $a_i$. So (i) holds. \\
 $(ii)$: If there is a vertex $x\in Y_3$ such that $xb_i\notin E$, then by (R\ref{Yi}:b), $xa_i\in E$, and then $\{a_i,v_3,v_4,b_i,x\}$ induces a paraglider which is a contradiction. So $[Y_3, \{b_i\}]$ is complete. Likewise, $[Y_4, \{a_i\}]$ is complete. Thus (ii) holds.\\
 $(iii)$: Suppose to the contrary that there are adjacent vertices $x\in Y_3$ and $y\in Y_5\sm \{b_i\}$.
  By item~$(i)$, we have $yb_i\in E$, and by item~$(ii)$, $xb_i\in E$. Now, $\{x,y,v_1,v_2,b_i\}$ induces a paraglider which is a contradiction. So (iii) holds. \\
 $(iv)$: This follows by item~$(i)$, (R\ref{AYi}:a) and  (R\ref{AYi}:b).
 \end{proof2}

Next we have the following:
 \begin{claim}\label{Y1Y3Y4}
 $Y_1$ is a clique, $|Y_3|\le 1$, $|Y_4|\le 1$, and $[Y_1, Y_3\cup Y_4]$ is complete.
 \end{claim}

\begin{proof2} First, since $y_2\in Y_2$ and $y_5\in Y_5$, by (R\ref{Yi}:d), $Y_1$ is a clique.
 Next, we know by Claim~\ref{Y2Y5}(ii) that $[Y_3,\{b_1\}]$ and $[Y_4, \{a_1\}]$ are  complete. So by (R\ref{Yi}:e), we have $|Y_3| \leq 1$ and  $|Y_4|\leq 1$. Finally, suppose to the contrary that there are non-adjacent vertices $y_1\in Y_1$ and $y_3\in Y_3$. We know by (R\ref{Yi}:b) that $y_1b_1\in E$, and by Claim~\ref{Y2Y5}(ii) that $y_3b_1\in E$. Now $\{b_1,v_1,v_2,y_3,y_1\}$ induces a paraglider, a contradiction. So $[Y_1, Y_3]$ is complete. Likewise, $[Y_1, Y_4]$ is  complete. \end{proof2}

 \smallskip
 For each $i\in \{1,2,\ldots, k\}$, let $A_i':= \{x\in A\mid x \mbox{~is~anti-complete~to~}\{a_i,b_i\}$ $ \mbox{and~is~complete~to~}
 (Y_2\sm \{a_i\} \cup (Y_5\sm \{b_i\})\}$. Let $A':=A_1'\cup \cdots \cup A_k'$. Let $A'':= \{x\in A\mid x \mbox{~is~complete~to}$ $\{a_1,\dots, a_k, b_1,\dots,b_k\}\}$.
Then by Claim~\ref{Y2Y5}(iv), $A= A'\cup A''$.
Moreover, we have the following claim.

\begin{claim}\label{AY134} The following hold:
(i) For each $i\in \{1,2,\ldots, k\}$, $|A_i'|\le 1$. (ii)~$[A',$ $Y_1\cup Y_3\cup Y_4] =\es$. (iii) $[A'', Y_1\cup Y_3\cup Y_4]$ is complete.
\end{claim}
\begin{proof2} $(i)$:~Suppose to the contrary that $|A_i'|\geq 2$ and let $x, x'\in A_i'$. Since $A_i' \subseteq A$ and $A$ is a clique (by (R\ref{AneiA}:a), $xx'\in E$. But, then $\{x,v_1,a_i,v_3,x'\}$ induces a paraglider. So (i) holds.  \\
$(ii)$: Suppose not. Then there are adjacent vertices $q\in A'$ and $y\in Y_1\cup Y_3\cup Y_4$.  Since $q\in A'$, there exists a pair $\{a_i,b_i\}$ such that $q$ is anti-complete to $\{a_i,b_i\}$. Now if $y\in Y_1$, then by (R\ref{Yi}:b), $yb_i\in E$, and then $\{v_1,q,v_4,b_i,y\}$ induces a paraglider which is a contradiction. So $y\in Y_3\cup Y_4$. Then  since  $[Y_3,\{b_i\}]$ and $[Y_4,\{a_i\}]$ are complete (by Claim~\ref{Y2Y5}(ii)), we have  a  contradiction to (R\ref{AYi}:a). So (ii) holds.\\
$(iii)$: Suppose not. Then there are non-adjacent vertices $q\in A''$ and $y\in Y_1\cup Y_3\cup Y_4$. Since $q\in A''$, $q$ is complete to $\{a_1,b_1\}$. Now if $y\in Y_1$, then by (R\ref{Yi}:b), $yb_1\in E$, and then $\{v_5,y,v_2,q,b_1\}$ induces a paraglider which is a contradiction. So $y\in Y_3\cup Y_4$. Then  since  $[Y_3,\{b_1\}]$ and $[Y_4,\{a_1\}]$ are complete (by Claim~\ref{Y2Y5}(ii)), we have  a  contradiction to (R\ref{AYi}:a). So (iii) holds.
 \end{proof2}

\begin{claim}\label{A2Y2Y5} Let $x\in A''$. Then either  $x$ is complete to $Y_2$ or  $x$ is complete to $Y_5$.
\end{claim}
\begin{proof2}Suppose not. Then  there exist vertices $p\in Y_2'$ and  $q\in Y_5'$ such that $xp,xq\notin E$. But, then $p$-$v_2$-$x$-$v_5$-$q$ is a $P_5$. \end{proof2}

\smallskip
Suppose that $A''\neq \es$, and let $x\in A''$. By Claim~\ref{A2Y2Y5} and up to symmetry, we may assume that $x$ is complete to $Y_5$. Then using (R\ref{AneiA}) and Claim~\ref{AY134} and since $x$ is not universal, we conclude that $x$ has a non-neighbor in $Y_2$. Moreover, by (R\ref{AYi}:d), $\overline{N(x)}\cap Y_2$ is a stable set. Now let us define $S:= \{x\} \cup (\overline{N(x)}\cap Y_2)$. Then since $[\{x\}, V(G)\sm (\overline{N(x)}\cap Y_2)]$ is complete, by (R\ref{Yi}:a), we see that $S$ is a good stable set of $G$. So, we may assume that $A''=\es$.

 \begin{claim}\label{Y2Y5-Z} The following hold:  (i) If $Y_2^*\cup Y_5^*\neq \es$, then $Z_1=\es$. (ii) If  $Y_2^*\neq \es$, then  $Y_3=\es$. (iii) If $Y_5^* \neq \es$, then $Y_4=\es$.\end{claim}
 \begin{proof2} To prove the claim, we show that if $Y_2^*\neq \es$, then $Z_1\cup Y_3=\es$, and the other cases follow by symmetry. Let $x\in Y_2^*$.
 Suppose to the contrary that $Z_1\cup Y_3\neq \es$, and let $y\in Z_1\cup Y_3$. We know by Claim~\ref{Y2Y5}(i) that $a_1x\in E$ and $b_1x\notin E$. Now if $y\in Z_1$, then  by (R\ref{ZiYj}:b), we have $ya_1,yb_1,yx\notin E$. But then $y$-$v_5$-$b_1$-$a_1$-$x$ is a $P_5$. So $y\in Y_3$. Then by (R\ref{Yi}:b), $xy\in E$, and by Claim~\ref{Y2Y5}(ii), $yb_1\in E$. But then $\{x,y,b_1,v_1,v_2\}$ induces a paraglider which is a contradiction. So, $Z_1\cup Y_3=\es$.
 \end{proof2}

\begin{claim}\label{Y1leq1}
 \mbox{If $Z_1\cup Z_2\cup Z_5 \cup A_1'\neq \es$, then $|Y_1|\leq 1$.}\end{claim}
\begin{proof2} Suppose not. Let $x,y\in Y_1$ and $z\in Z_1\cup Z_2\cup Z_5 \cup A_1'$. We know by Claim~\ref{Y1Y3Y4} that $Y_1$ is a clique and so $xy\in E$. Moreover, by (R\ref{Yi}:b), $\{x,y\}$ is complete to $\{a_1,b_1\}$.  If $z\in Z_1$, then $zx,xy\in E$ (by (R\ref{ZiYj}:a)), and then by (R\ref{ZiYj}:b), $\{x,z,v_3,a_1,y\}$ induces a paraglider.
 If $z\in Z_2\cup Z_5$, then $zx,xy\notin E$ (by (R\ref{ZiYj}:b)). Also if $z\in Z_2$, then $za_1\in E$ (by (R\ref{ZiYj}:a)), and if  $z\in Z_5$, then $zb_1\in E$ (by (R\ref{ZiYj}:a)). But then either $\{x,v_5,z,a_1,y\}$ or  $\{x,v_2,z,b_1,y\}$ induces a paraglider.
 If $z\in A_1'$, then by Claim~\ref{AY134}(ii), $zx,zy\notin E$. But then $\{x,v_2,z,v_5,y\}$ induces a paraglider.
\end{proof2}

Suppose that $Y_2^* \cup Y_5^*=\es$. Since $A''=\es$, $A=A_1'$.
If $Y_3\cup Y_4\cup Z_1 \cup  Z_2\cup Z_5 \cup A_1'\neq \es$ or if $|Y_1|\le 1$, then by Claims~\ref{Y1Y3Y4} and \ref{Y1leq1}, and (R\ref{Yi}:e) we conclude that  $Y$ is a clique with $|Y|\le 5$. So $|V(G)|=|V(C_5)|+ |Z|+|Y|+|A_1'|\le  16$, and we see that in this case, $G$ is an induced subgraph of the complement of the Clebsch graph.
   If $Y_3\cup Y_4\cup Z_1 \cup  Z_2\cup Z_5 \cup A_1'=\es$ and $|Y_1|\ge 2$, then since $Y_1$ is complete to $\{a_1,b_1\}$ (by (R\ref{Yi}:b)), we see that $\omega(G)\geq 5$, and hence $\{v_1\}$ is a good stable set of $G$.

So suppose that $Y_2^* \cup Y_5^*\neq \es$. We may assume, up to symmetry, that $Y_2^*\neq \es$. Let  $p\in Y_2^*$.  Then by Claim~\ref{Y2Y5-Z}, $Z_1\cup Y_3 =\es$. Further we have the following.
\begin{claim}\label{Y1-Z2Z5A'}
 \mbox{We have: Either $Y_1=\es$ or $Y_4\cup Z_2\cup Z_5\cup A'=\es$.}\end{claim}
 \begin{proof2}Suppose not. Let $x\in Y_1$ and $y\in Y_4\cup Z_2\cup Z_5\cup A'$. Then by (R\ref{Yi}:b), $xb_1, px\in E$.  Now:
  If $y\in Y_4$, then $yx\in E$ ( by Claim~\ref{Y1Y3Y4}), $py\notin E$ (by Claim~\ref{Y2Y5}(iii)). But then $\{x,y,v_3,p,v_2\}$ induces a paraglider.
 If $y\in Z_2$, then by (R\ref{ZiYj}), $yb_1, yp\in E$ and $yx\notin E$. But then $\{p,x,v_5,y,b_1\}$ induces a paraglider.
 If $y\in Z_5$, by (R\ref{ZiYj}), $yb_1, yp\in E$ and $yx\notin E$. But, then $\{p,x,b_1,y,v_2\}$ induces a paraglider.
So, we may assume that $y\in A'$. Then there exists a pair $\{a_i,b_i\}$ such that $y$ is anti-complete to $\{a_i,b_i\}$. By Claim~\ref{Y2Y5}(i), $pa_i\in E$ and $pb_i\notin E$. Then since $p$-$v_2$-$y$-$v_5$-$b_i$ is not a $P_5$, we have $py\in E$. Also, by Claim~\ref{AY134}, $xy\notin E$. But now $\{p,x,v_5,y,v_2\}$ induces a paraglider which is a contradiction. So the claim holds.
\end{proof2}

First suppose that $Y_1\neq \es$. Then by Claim~\ref{Y1-Z2Z5A'}, $Y_4\cup  Z_2\cup Z_5\cup A'=\es$.  But then $\{v_1\}$ is a good stable set of $G$. 
 So, we may assume that $Y_1=\es$.

Next suppose that $Y_4\neq \es$. Then by Claim~\ref{Y1Y3Y4}, we let $Y_4=\{y_4\}$. Also, by Claim~\ref{Y2Y5-Z}, $Y_5^*=\es$.  Moreover, we show that $Z_3=\es$. Suppose not, and let $z_3\in Z_3$. Then by (R\ref{ZiYj}:b), $z_3p,z_3y_4\notin E$. But then $y_4$-$v_4$-$z_3$-$v_1$-$p$ is a $P_5$. So, $Z_3=\es$. Now, in this case, we see that there is a good stable set of $G$ as follows:
 If $Z=\es$, then $\{y_4,v_2\}$ is a good stable set.  So $Z\neq \es$.
  If $Z_2\neq \es$, then by (R\ref{Zi}), (R\ref{ZiYj}:a), and (R\ref{AneiA}),  $[Z_2, V(G)\sm \{v_2\}]$ is complete, and hence  $\{z_2,v_2\}$ is a good stable set of $G$. So $Z_2=\es$.
 Next if $Z_4\neq \es$, then since $[Z_4, V(G)\sm \{v_4,b_1\}]$ is complete,  $\{z_4,v_4\}$ is a good stable set of $G$. So $Z_4=\es$.
  Finally, if $Z_5\neq \es$, then since $[Z_4, V(G)\sm \{b_1\}\cup Y_4]$ is complete, $\{z_5,v_5\}$ is a good stable set of $G$.

So, we may assume that $Y_4=\es$. If $Z_2\neq \es$,  then since   $[Z_2, V(G)\sm (\{v_2\}\cup Z_3)]$ is complete, we see that $\{z_2,v_2\}$ is a good stable set of $G$. So $Z_2=\es$. Likewise, $Z_5=\es$.
 Then we define $Q_1\cup R_1:= \{a_1,a_2,\ldots,a_k, v_2,v_3\}\cup Y_2'\cup Z_4$, $Q_2\cup R_2:= \{b_1,b_2,\ldots,b_k,v_4,v_5\}\cup Y_5'\cup Z_3$,  and $S:= \{v_1\}\cup A'$. Now, it is easy see that $G\in \cal{H}$.

 This completes the proof of the theorem.
\end{proof}

\begin{theorem}\label{thm:P5PF1F2-free-F3}
Let $G$ be a ($P_5,F_1,F_2$,paraglider)-free atom with no universal or pair of comparable vertices. Suppose that $G$ contains  $F_3$. Then $G$ has a stable set $S$  such that either $S$ is good or $G\setminus S$ is  perfect.
\end{theorem}
\begin{proof} Let $G$ be the given graph with vertex-set $V$ and edge-set $E$.
Consider the graph $F_3$ as shown in Figure~\ref{fig:GF123} and let
$C=\{v_1,\ldots, v_5\}$.  We use the same notation as in
Theorem~\ref{thm:C5} and use the properties in Theorem~\ref{thm:C5}.  So $y_1 \in Y_1$ and $z_1\in Z_1$. Since $G$ is $F_1$-free, $T\cup X=\es$ (by (R\ref{F1freeXemp})). Moreover, we have the following:

\begin{claim}\label{F3-YiYi+2} For each $i$, we have  $[Y_{i}, Y_{i+2}]=\es$.
\end{claim}
\begin{proof2} Suppose to the contrary that there are adjacent vertices, say $x\in Y_i$ and $y\in Y_{i+2}$. Then $\{v_1,v_2,v_3,v_4,v_5,x,y\}$ induces an $F_2$ which is a contradiction. So the claim holds. \end{proof2}

\begin{claim}\label{F3-ZY} For each $i$, we have either $Z_i =\es$ or $Y_{i-1}\cup Y_{i+1}=\es$.
\end{claim}
\begin{proof2} Suppose not. Up to symmetry, let $z\in Z_i$ and $y\in Y_{i+1}$. Then by (R\ref{ZiYj}), $zy\notin E$. But, then $\{v_1,v_2,v_3,v_4,v_5,z,y\}$ induces an $F_2$. \end{proof2}

\medskip

Since $z_1\in Z_1$ and $y_1\in Y_1$, by Claim~\ref{F3-ZY}, the sets $Y_2, Y_5, Z_2$ and $Z_5$ are empty. If $A=\es$, then  up to symmetry, we  have three cases: (a) $Z_3\neq \es$ and $Z_4 \neq \es$. (b) $Z_3\neq \es$ and $Z_4=\es$. (c) $Z_3 =\es$ and $Z_4 = \es$.
In Case (a), $Y_3 =Y_4 = \es$, and hence $G- \{v_2, v_4\}$ is a $\overline{P_3}$-free expansion of a perfect graph, and hence perfect. In Case~(b) and in Case~(c), $G-\{v_2, v_5\}$ is a $\overline{P_3}$-free  expansion of a $P_4$, and hence perfect.

So suppose that $A\neq \es$. First let us assume that $[A, Y_1]$ is not complete. Then there exists a vertex $x \in A$ that has a non-neighbor in $Y_1$, say $y\in Y_1$. If $x$ has a non-neighbor $y'\in Y_3\cup Y_4$, then $y$-$v_1$-$x$-$v_3$-$y'$ or $y$-$v_1$-$x$-$v_4$-$y'$ is a $P_5$. So,  by (R\ref{AneiA}), $[\{x\}, V(G)\sm Y_1]$ is complete. Then by (R\ref{AYi}:d), $\{x\}\cup (\overline{N(x)} \cap Y_1)$ is a good stable set of $G$. So we may assume that $[A, Y_1]$ is   complete. Then since $G$ has no universal vertex, by (R\ref{AneiA}), $x$ has a non-neighbor in $Y_3\cup Y_4$.  Then $\{x\}\cup (\overline{N(x)} \cap Y_3)$ or $\{x\}\cup (\overline{N(x)}\cap Y_4)$ is a good stable set of $G$. This complete the proof of the theorem.
\end{proof}

\begin{theorem}\label{thm:P5PF123-free-C5}
Let $G$ be a ($P_5,F_1,F_2, F_3$,paraglider)-free atom with no universal or pair of comparable vertices. Suppose that $G$ contains a $C_5$. Then one of the following hold:
\begin{itemize}\itemsep=0pt
\item $G$ is an induced subgraph of the complement of the Petersen graph.
\item $G$ is an $\overline{P_3}$-free expansion of $C_5$.
\item   $G$ has a stable set $S$ such that $G\sm S$ is perfect.
\end{itemize}

\end{theorem}
\begin{proof} Let $G$ be the given graph with vertex-set $V$ and edge-set $E$.
 Suppose that $G$ contains  $C_5$ with vertex set $C=\{v_1,\ldots, v_5\}$.  We use the same notation as in
Theorem~\ref{thm:C5} and use the properties in Theorem~\ref{thm:C5}. Since $G$ is $F_1$-free, $T\cup X =\es$ (by (R\ref{F1freeXemp})).
Since $G$ is $F_2$-free, we have, for each $i$, $[Y_i, Y_{i+2}]=\es$. Since $G$ is $F_3$-free, $[A, Y]$ is complete. So by (R\ref{AneiA}), any vertex in $A$ is a universal vertex of $G$  and hence $A=\es$.  Moreover, if $z_i\in Z$, then  since $G$ is ($F_2,F_3$)-free, by (R\ref{ZiYj}), we have $Y_j=\es$, for $j\in\{i-1,i, i+1\}$.

If $Z=\es$, then $G$ is a $\overline{P_3}$-free expansion of $C_5$ (by (R\ref{Yi})). So let us assume that $Z\neq \es$.
If there exists an  $i$ such that $z_i, z_{i+2}\in Z$, then $Y=\es$. Now  by (R\ref{Zi}), $G$ is an induced subgraph of the complement of the Petersen graph. Finally up to symmetry, let us assume that   $Z= \{z_1\}$ or $Z=\{z_1,z_{2}\}$. (a) If $Z=\{z_1\}$, then $Y_j=\es$, for $j\in\{1,2,5\}$, and by (R\ref{ZiYj}), $[Y_3\cup Y_4, Z]$ is complete. So, by (R\ref{Yi}:a), $G\sm \{v_1\}$ is  a $\overline{P_3}$-free expansion of a perfect graph, and hence perfect. (b) If $Z=\{z_1,z_2\}$, then $Y_j=\es$, for $j\in\{1,2,3, 5\}$, and by (R\ref{ZiYj}), $[Y_4, Z]$ is complete. Then we see that by (R\ref{Yi}:a), $G\sm \{v_1,v_3\}$ is  a $\overline{P_3}$-free expansion of a perfect graph, and hence perfect. This completes the proof of the theorem.
\end{proof}

\section{Coloring ($P_5$,\,paraglider)-free graphs}
Given a graph $G$ and a proper homogeneous set $X$ in $G$, let $G/X$
be the graph obtained by replacing $X$ with a clique $Q$ of size
$\omega(X)$ (i.e., $G/X$ is obtained from $G\setminus X$ and $Q$ by
adding all edges between $Q$ and the vertices of $V(G)\setminus X$
that are adjacent to $X$ in $G$).
 The proof of the following lemma is very similar to that of Lemma~3.1  of \cite{KM2018} and we omit the details.
\begin{lemma} [\cite{KM2018}]\label{lem:red}
In a graph $G$ let $X$ be a proper homogeneous set such that $G[X]$ is
perfect.  Then $\omega(G)=\omega(G/X)$ and  $\chi(G)=\chi(G/X)$.  \hfill{$\Box$}  
\end{lemma}

 Let ${\cal C}_5$ be the class of
graphs that are $\overline{P_3}$-free expansions of $C_5$, and let ${\cal C}_5^*$
be the class of graphs that are clique expansions of $C_5$.

Let $\cal{H}^*$ be the class of graphs $G \in \cal{H}$ such that, with the notation as in Section~1, the two sets $R_1$ and $R_2$ are cliques.

 Since  $\overline{P_3}$-free graphs are perfect, the following lemma (Lemma~\ref{lem:red2}) can be proved using Lemma~\ref{lem:red}, and the proof is very similar to that of Lemma~3.3  of \cite{KM2018}, so we omit the details.

\begin{lemma}[\cite{KM2018}] \label{lem:red2}
For every graph $G$ in ${\cal C}_5$ (resp. $G$ in $\cal{H}$)
 there is a graph $G^*$ in ${\cal C}_5^*$
 (resp. $G^*$ in $\cal{H}^*$) such that
$\omega(G)=\omega(G^*)$ and $\chi(G)=\chi(G^*)$.
\hfill{$\Box$}
\end{lemma}

\begin{lemma}[\cite{KM2018}]\label{lem:C5exp}
Let $G$ be a clique expansion of $C_5$. Then  $\chi(G)\le \lceil\frac{5\omega(G)}{4}\rceil$.  \hfill{$\Box$}
\end{lemma}

\enlargethispage{1cm}
For any fixed integer $k\ge 2$, let $G_k$ be the graph defined as follows.

\begin{itemize}\itemsep=1pt
\item $V(G_k)$ can be partitioned into three cliques $Q_1:=\{a_1,a_2\ldots, a_k\}$, $Q_2:=\{b_1,b_2,\ldots,b_k\}$,
and $S:=\{s_1,s_2,\ldots,s_k\}$ such that $[Q_1, Q_2]$ is a perfect matching, say  $\{a_1b_1,a_2b_2,\ldots,a_kb_k\}$.
\item  There exists an injective function $f:S\rightarrow \{1,2,\ldots, k\}$ such that for each vertex  $x\in S$, $\{x\}$ is anti-complete to $\{a_{f(x)}, b_{f(x)}\}$, and is complete to $(Q_1\cup Q_2) \sm \{a_{f(x)}, b_{f(x)}\}$.
\item No other edges in $G$.
\end{itemize}

\newpage
\begin{lemma}\label{lem:Gk}
For each integer $k\ge 2$, $\chi(G_k)\le \lceil \frac{3k}{2} \rceil$.
\end{lemma}

\begin{proof}
Without loss of generality, we may assume that for each $i\in \{1,2\ldots,k\}$,
$s_i$ is anti-complete to $\{a_i,b_i\}$ and complete to $(Q_1\cup Q_2)\sm \{a_i,b_i\}$.
We consider two cases depending on whether $k$ is even or not.

Suppose first that $k=2t$ for some $t\ge 1$. Now we color $G_{2t}$ using $3t$ colors as follows:

$\bullet$ Color $s_1,s_2,\ldots,s_t$ with colors $1,2,\ldots,t$, respectively.

$\bullet$ Color $s_{t+1},s_{t+2},\ldots,s_{2t}$ with colors $t+1,t+2,\ldots,2t$, respectively.

$\bullet$ Color $a_1,a_2,\ldots,a_t$ with colors $2t+1,2t+2,\ldots,3t$, respectively.

$\bullet$ Color $a_{t+1},a_{t+2},\ldots,a_{2t}$ with colors $t+1,t+2,\ldots,2t$, respectively.

$\bullet$ Color $b_1,b_2,\ldots,b_t$ with colors $1,2,\ldots,t$, respectively.

$\bullet$ Color $b_{t+1},b_{t+2},\ldots,b_{2t}$ with colors $2t+1,2t+2,\ldots,3t$, respectively.

\noindent Then it can be easily checked that the above  is  a $3t$-coloring of $G_{2t}$.

Next suppose that $k=2t+1$, for some $t\ge 1$. Then $\lceil\frac{3k}{2}\rceil=3t+2$.
Observe that $G_{2t+1}\sm \{s_{2t+1},a_{2t+1},b_{2t+1}\}$ is isomorphic to $G_{2t}$.
Therefore, $\chi(G_{2t+1})\le 3t+2$.
\end{proof}

\begin{theorem}\label{thm:H}
If $G\in \mathcal{H}^*$, $\chi(G)\le \frac{3\omega(G)}{2}$.
\end{theorem}

\begin{proof}
Let $G$ be partitioned into $Q_1$, $Q_2$, $R_1$, $R_2$ and $S$.
Without loss of generality, we may assume that for each $1\le i\le |S|$,
$s_i$ is complete to $(Q_1\cup Q_2)\setminus \{a_i,b_i\}$ and is anti-complete to $\{a_i,b_i\}$.
Let $r=\max\{|R_1|,|R_2|\}$.
Since $R_i\cup Q_i$ is a clique for $i\in \{1,2\}$,
$\omega(G)\ge k+r$.
Obviously,
\[
	\chi(G)\le \chi(G[Q_1\cup Q_2\cup S])+\chi(G[R_1\cup R_2]).
\]
Since $|S|\le k$, $G[Q_1\cup Q_2\cup S]$ is an induced subgraph of $G_k$.
By Lemma~\ref{lem:Gk}, $\chi(G[Q_1\cup Q_2\cup S])\le \lceil \frac{3k}{2} \rceil$.
On the other hand, since $R_1$ and $R_2$ are clique, $\chi(G[R_1\cup R_2])\le r$.
If $r\ge 1$, then $r+\frac{1}{2}\le \frac{3r}{2}$. Therefore,
$\chi(G)\le (\frac{3k}{2}+\frac{1}{2})+r\le \frac{3k}{2}+\frac{3r}{2}\le \frac{3\omega(G)}{2}$.
So we may assume that $r=0$.  Observe that $\omega(G)\in \{k,k+1\}$ and that if $\omega(G)= k$, then $|S|\le k-2$. If $S=\es$, then $\chi(G)\le k \le \omega(G)$.
So, let $S=\{s_1,s_2,\ldots, s_t\}$ for some $1\le t\le k$.
Let
\begin{equation*} \label{eq1}
\begin{split}
G' & = G[\{a_1,a_2,\ldots a_t\}\cup \{b_1,b_2,\ldots b_t\}\cup \{s_1,s_2,\ldots,s_t\}], \mbox{and} \\
G''  & = G[\{a_{t+1},\ldots, a_k\}\cup \{b_{t+1},\ldots, a_k\}]. \\
\end{split}
\end{equation*}

Observe that $\chi(G'')=k-t$. Since $G'$ is isomorphic to $G_{t}$, it follows from Lemma~\ref{lem:Gk} that
$\chi(G')\le \lceil \frac{3t}{2} \rceil\le \frac{3t}{2}+\frac{1}{2}$. Therefore,
 $\chi(G)\le \chi(G')+\chi(G'')\le (\frac{3t}{2}+\frac{1}{2})+(k-t)=k+\frac{t}{2}+\frac{1}{2}$. Now if $\omega(G)=k$, then since $t\leq k-2$, we have $\chi(G)\le k+\frac{k-2}{2}+\frac{1}{2}= \frac{3k}{2}-\frac{1}{2}<\frac{3}{2}\omega(G)$, and if $\omega(G)=k+1$, then $\chi(G)\le k+\frac{k}{2}+\frac{1}{2}< \frac{3}{2}\omega(G)$.
\end{proof}

Let $G$ be a graph and $v\in V(G)$. We say that $G'$ is obtained from $G$ by {\em adding a smaller vertex} $u$
if   $N(u)$ is a non-empty subset of $N(v)$ in $G'$. Let $\mathcal{B}$ be the set of graphs that
consists of the complement of the Clebsch graph and the graph obtained from the complement of the Clebsch graph by deleting a vertex.
We note that any graph $G\in \mathcal{B}$ has $\chi(G)=\lceil \frac{3\omega(G)}{2} \rceil$
and the ceiling is necessary. Moreover, it is not hard to verify that
these are the only induced subgraphs of the complement of the Clebsch graph that satisfy this property (This fact and Lemma \ref{lem:basegraphs} below are verified by a computer program
due to Owen Merkel).
Let $\mathcal{G}$ be the class of ($P_5$,\,paraglider)-free graphs that can be obtained from a graph in $\mathcal{B}$
by a sequence of adding a smaller vertex.  We say that a graph $G$ is {\em awesome} if for every non-empty clique $K$ of $G$, there exists
an induced $P_4:=$ $v$-$x$-$y$-$z$ such that $v\in K$ and $x,y,z\notin K$.
Then we have the following lemma and its proof is verified easily by a computer program.

\begin{lemma}\label{lem:basegraphs}
Every graph $G\in \mathcal{B}$ is awesome.
\end{lemma}

\begin{lemma}\label{lem:P4}
Every graph $G\in \mathcal{G}$ is awesome.
\end{lemma}

\begin{proof}
Let $G\in \mathcal{G}$. Then $G$ is obtained from a graph $B\in \mathcal{B}$ by adding smaller vertices $u_1,\ldots,u_k$ sequentially.
We prove the lemma by induction on $k$. If $k=0$, then the lemma holds by Lemma \ref{lem:basegraphs}.
Suppose now that the lemma holds for all graphs in $\mathcal{G}$ that are obtained from  a graph $B\in \mathcal{B}$ by adding $k-1$ vertices
for some $k\ge 1$. Let $G'=G-u_k$. By the inductive hypothesis, $G'$ is awesome, i.e., for every non-empty clique
$K$ of $G'$ there exists an induced $P_4$ $:=$ $v$-$x$-$y$-$z$ such that $v\in K$ and $x,y,z\notin K$.

Now let $K$ be a clique of $G$. If $K\setminus \{u_k\}$ is non-empty, then it follows from the inductive hypothesis
that a desired $P_4$ exists for $K$. It remains to consider the case that $K=\{u_k\}$, i.e., to show
that in $G$ there is an induced $P_4$ $:=$ $u_k$-$x$-$y$-$z$.
Let $v\in V(G')$ be the vertex such that $N_G(u_k)\subseteq N_G(v)$.

Suppose first that $N_G(u_k)=N_G(v)$. Since $G'$ is awesome, there exists an induced path  $v$-$x$-$y$-$z$ in $G'$ (consider the clique $\{v\}$).
Thus, $u_k$-$x$-$y$-$z$ is a desired $P_4$.

So, we may assume that there exists a vertex $d\in V(G')$ such that $d$ is adjacent to $v$ but not to $u_k$.
Suppose that $u_k$ has two non-adjacent neighbors $s$ and $t$ in $G$. Since $\{v,s,u_k,t,d\}$ does not induce a paraglider,
$d$ is not adjacent to either $s$ or $t$, say $s$. Then $u_k$-$s$-$v$-$d$ is a desired $P_4$.

So, $N_G(u_k)$ is a clique. Let $w\in V(G')$ be a neighbor of $u$. Since $G'$ is awesome, there exists an induced $P_4$ $:=$ $w$-$x$-$y$-$z$ in $G'$.
Since $N_G(u_k)$ is a clique, $u_k$ is adjacent to neither $y$ nor $z$. If $u_k$ is not adjacent to $x$, then $u_k$-$w$-$x$-$y$ is a desired $P_4$.
Otherwise $u$ is adjacent to $x$ and $u_k$-$x$-$y$-$z$ is a desired $P_4$.
\end{proof}

We are now ready to prove the main theorem in this section.

\begin{theorem}\label{thm:chibound}
Let $G$ be a connected ($P_5$,\,paraglider)-free graph. Then $\chi(G)\le \lceil \frac{3\omega(G)}{2} \rceil$.
Moreover, $\chi(G)>\frac{3\omega(G)}{2}$ if and only if $G\in \mathcal{G}$.
\end{theorem}

\begin{proof}
Observe that every graph $H$ in $\mathcal{G}$ has $\chi(H)=8$ and $\omega(H)=5$ and so $\chi(H)=\lceil \frac{3\omega(H)}{2} \rceil$.
Therefore, if $G\in \mathcal{G}$ then $\chi(G)>\frac{3\omega(G)}{2}$.
We now show by induction on $|V(G)|$ that if  $G\notin \mathcal{G}$, then $\chi(G)\le \frac{3\omega(G)}{2}$.
This will imply the theorem.

First, suppose that $G$ contains a pair of comparable vertices $u$ and $v$, say $N(u)\subseteq N(v)$.
Then since $G\notin \mathcal{G}$, it follows that $G-u\notin \mathcal{G}$.
Moreover, $G-u$ is connected. By the inductive hypothesis,
$\chi(G-u)\le \frac{3\omega(G-u)}{2}$. Note that $\chi(G)=\chi(G-u)$ and $\omega(G)=\omega(G-u)$.
Therefore, $\chi(G)\le \frac{3\omega(G)}{2}$.

Suppose now that $G$ contains a clique cutset. Let $K$ be a minimal clique cutset and $G\sm K$ is the disjoint union of two
subgraphs $H_1$ and $H_2$. Let $G_i=G[K\cup V(H_i)]$ for $i\in \{1,2\}$. Note that $G_1$ and $G_2$ are connected.
We show that  neither $G_1$ nor $G_2$ is in $\mathcal{G}$.
Suppose not. We may assume by symmetry that $G_1\in \mathcal{G}$.
Since $K$ is a non-empty clique of $G_1$, it follows from Lemma~\ref{lem:P4}
that there exists an induced $P_4$ $:=$ $v$-$x$-$y$-$z$ with $v\in K$ and $x,y,z\notin K$.
Since $K$ is minimal, $v$ has a neighbor $w$ in $G_2$. Then $w$-$v$-$x$-$y$-$z$ is a $P_5$.
This contradicts the fact that $G$ is $P_5$-free.
Since neither $G_1$ nor $G_2$ is in $\mathcal{G}$, $\chi(G_i)\le \frac{3\omega(G_i)}{2}$ by the inductive hypothesis.
Therefore, $\chi(G)=\max\{\chi(G_1),\chi(G_2)\}\le \frac{3\omega(G)}{2}$.

Suppose that $G$ contains a universal vertex $u$. If $G-u$ is disconnected, then $\{u\}$ is a clique cutset of $G$ and we are done by the argument on clique cutsets. Therefore, $G-u$ is connected.
If $G-u\notin \mathcal{G}$, then
it follows from the inductive hypothesis that
$\chi(G-u)\le \frac{3\omega(G-u)}{2}+1=\frac{3(\omega(G)-1)}{2}+1<\frac{3\omega(G)}{2}$.
If $G-u\in \mathcal{G}$, then since every graph in $\mathcal{G}$ has clique number 5 and chromatic number 8,
it follows that $\chi(G)=9$ and $\omega(G)=6$. Thus, $\chi(G)=\frac{3\omega(G)}{2}$.

Therefore, we may assume that $G$ contains no clique cutsets, universal vertices or pairs of comparable vertices.
We now can apply the structure theorem.

If $G$ is an induced subgraph of the complement of the Clebsch graph,
then the theorem clearly holds.

If $G$ is a $\overline{P_3}$-free expansion of $C_5$, then it follows by Lemma~\ref{lem:C5exp} that $\chi(G)\le \lceil \frac{5\omega(G)}{4} \rceil
\le \frac{3\omega(G)}{2}$.

Suppose that $G$ has a good stable set $S$, and that $H_1,\ldots,H_t$ are the component of $G\sm S$, where $t\ge 1$.
Since $S$ is good,  it follows that $\omega(H_i)\le \omega(G)-1$.
We claim that $\chi(H_i)\le \frac{3}{2}\omega(G)-1$ for each $i\in \{1,\ldots,t\}$.
If $H_i\notin \mathcal{G}$,  then it follow from the inductive hypothesis that
$\chi(H_i)\le  \frac{3\omega(H_i)}{2}\le \frac{3(\omega(G)-1)}{2}< \frac{3\omega(G)}{2}-1$.
If $H_i\in \mathcal{G}$, then $\omega(H_i)=5$ and  $\chi(H_i)=8$. This implies that $\omega(G)\ge 6$.
Thus, $\chi(H_i)=8=\frac{3}{2}\times 6-1\le \frac{3}{2}\omega(G)-1$.
Therefore, $\chi(G)\le \max_{1\le i\le t}\chi(H_i)+1\le \frac{3}{2}\omega(G)$.

If $G$ has a stable set $S$ such that $S$ is perfect,
then $\chi(G)\le \chi(G\sm S)+1=\omega(G\sm S)+1\le \omega(G)+1\le \frac{3\omega(G)}{2}$.

If $G\in \mathcal{H}$, then the theorem follows from Lemma~\ref{lem:red2} and Theorem~\ref{thm:H}.
\end{proof}

The following construction shows that
our bound in Theorem~\ref{thm:chibound} is tight up to an additive constant.
Consider the graph $G_k$, for $k\ge 2$ as defined earlier.  Note that $G_k\in \cal{H}$.
It is not hard to verify that $G$ is ($P_5$,\,paraglider)-free, $\alpha(G_k)=2$,
and  $\omega(G_k)=k+1$. Since $\chi(G_k)\geq \frac{|V(G_k)|}{\alpha(G_k)}$,
we have $\chi(G_k)\geq\frac{3k}{2}=\frac{3}{2}(\omega(G_k)-1)$.
When $k=2t+1$ for some integer $t\ge 1$, it follows from Lemma \ref{lem:Gk}
that $\chi(G_{k})=3t+2=\lceil \frac{3\omega(G_k)}{2} \rceil -1$.
This implies that there is no $(\frac{3}{2}-\epsilon)$-approximation algorithm for the chromatic number for ($P_5$,\,paraglider)-free graphs
for any $\epsilon>0$.

\medskip
Theorem~\ref{thm:chibound} has the following two corollaries.

\begin{corollary}
Every ($P_5$,\,paraglider)-free graph $G$ has $\chi(G)\le \lceil \frac{3\omega(G)}{2} \rceil$.
\end{corollary}

\begin{proof}
Let $G_1,\ldots,G_t$ with $t\ge 1$ be the components of $G$. Then by Theorem~\ref{thm:chibound},
$\chi(G_i)\le \lceil \frac{3\omega(G_i)}{2} \rceil\le \lceil \frac{3\omega(G)}{2} \rceil$.
Since $\chi(G)=\max_{1\le i\le t}\chi(G_i)$, the corollary follows.
\end{proof}

\begin{corollary}
Let $G$ be a ($P_5$,\,paraglider)-free graph (not necessarily connected). Then $\chi(G)>\frac{3\omega(G)}{2}$ if and only if
there exists a component $C$ of $G$ such that $C\in \mathcal{G}$ and $\omega(G)\le 5$.
\end{corollary}

\begin{proof}
Clearly, if $G$ has a component $C$ of $G$ such that $C\in \mathcal{G}$ and $\omega(G)\le 5$, then
$\chi(G)=8$ and $\omega(G)=5$ and so $\chi(G)>\frac{3\omega(G)}{2}$.

Conversely, suppose that $G$ does not satisfy the condition. We show that $\chi(G)\le \frac{3\omega(G)}{2}$.
Let $G_1,\ldots,G_t$ with $t\ge 1$ be the components of $G$. If none of $G_1,\ldots,G_t$ is in $\mathcal{G}$,
then by Theorem~\ref{thm:chibound}, $\chi(G_i)\le \frac{3\omega(G_i)}{2}\le \frac{3\omega(G)}{2}$ and so we are done.
So we may assume by symmetry that $G_1\in \mathcal{G}$. Then $\omega(G)\ge 6$.
For each $1\le i\le t$, we show that $\chi(G_i)\le \frac{3\omega(G)}{2}$.
If $G_i\in \mathcal{G}$, then $\chi(G_i)=8=\frac{3}{2}\times 6-1\le \frac{3}{2}\omega(G)-1$.
If $G_i\notin \mathcal{G}$,  then $\chi(G_i)\le \frac{3\omega(G)}{2}$ by Theorem~\ref{thm:chibound}.
This completes the proof.
\end{proof}

\noindent {\bf Acknowledgment.} We would like to thank Owen Merkel for writing a computer program to verify that
the complement of the Clebsch graph and
the graph obtained from the complement of the Clebsch graph by deleting a vertex are the only two induced subgraphs
of the complement of the Clebsch graph that satisfy $\chi=\lceil \frac{3\omega}{2} \rceil$ where the ceiling is necessary,
and Lemma~\ref{lem:basegraphs}.

\end{document}